\documentclass[pdftex]{amsart} 

\usepackage{array}
\usepackage{enumerate}
\usepackage[margin=2.0cm]{geometry}
\usepackage{amsfonts,amssymb,amsmath,latexsym,upref}
\usepackage{verbatim} 
\usepackage[english]{babel}
\usepackage[latin1]{inputenc}  
\usepackage{color,tikz,pgf}

\usetikzlibrary{shapes, intersections, calc, patterns, positioning,
  matrix, arrows, fit, backgrounds, 
  decorations.pathmorphing, decorations.markings}

\newtheorem{lemma}[equation]{Lemma}
\newtheorem{prop}[equation]{Proposition}
\newtheorem{thm}[equation]{Theorem}
\newtheorem{cor}[equation]{Corollary}

\newtheorem{defn}[equation]{Definition}
\newtheorem{conj}[equation]{Conjecture}

\theoremstyle{definition}
\newtheorem{exmp}[equation]{Example}

\newtheorem{rmk}[equation]{Remark}

\numberwithin{equation}{section}

\newcommand{\F}{\mathbf{F}}

\newcommand{\Aut}[2]{\operatorname{Aut}_{#1}(#2)}

\newcommand{\qt}[2]{#1 \backslash #2}
\newcommand{\isot}[2]{{}_{#1}#2}

\newcommand{\GL}[3][{}]{\operatorname{GL}^{#1}_{#2}(#3)}

\newcommand{\cat}[3]{\mathcal{#1}_{#2}^{#3}}  
  
\newcommand{\m}{morphism}

\newcommand{\gen}[1]{{\langle}#1{\rangle}}

\newcommand{\syl}[1]{Sylow $#1$-subgroup}
\newcommand{\we}{weighting}
\newcommand{\Mb}{M\"obius}
\newcommand{\Euc}{Euler characteristic}
\newcommand{\rchi}{\widetilde{\chi}}

\newcommand{\mynote}[1]{\noindent{\textcolor{red}{\textbf{[#1]}}}}

\title{Equivariant Euler characteristics of partition posets}

\author{Jesper M.~M\o ller}
\address{Institut for Matematiske Fag\\
  Universitetsparken 5\\
  DK--2100 K\o benhavn}
\email{moller@math.ku.dk}
\urladdr{htpp://www.math.ku.dk/~moller}

\thanks{Supported by the Danish National Research Foundation through
  the Centre for Symmetry and Deformation (DNRF92) and by Villum
  Fonden through the project Experimental Mathematics in Number
  Theory, Operator Algebras, and Topology} 

\usepackage[bookmarks=true,bookmarksopen=false]{hyperref}
\hypersetup{
   pdftitle = {},
   pdfauthor = {Jesper Michael Møller},
   pdfpagemode = {UseOutlines},
   pdfstartview = {FitH},
   pdfborder = {0 0 0},
   backref = {true},
   colorlinks = {true},
   urlcolor = {blue},
   citecolor = {blue},
   linkcolor = {blue},
   pdftoolbar = {true},
}


\begin{document}
\date{\today}
\begin{abstract}
  We compute all the equivariant \Euc s of the $\Sigma_n$-poset of partitions
  of the $n$ element set.
\end{abstract}
\maketitle


\section{Introduction}
\label{sec:intro}

Let $G$ be a finite group and $\Pi$ a finite $G$-poset. 
For $r \geq 1$, let $C_r(G)$ denote the set
of tuples $X=(x_1,\ldots,x_r)$ of $r$ commuting elements
of $G$. Write $\Pi^X$ for the subposet consisting of all elements of $\Pi$
fixed by all elements of $X \in C_r(G)$.
The $r$th reduced
equivariant \Euc\ of the $G$ poset $\Pi$, as defined by Atiyah and Segal
\cite{atiyah&segal89}, is the normalized sum
\begin{equation*}
  \rchi_r(\Pi,G) = \frac{1}{|G|} \sum_{X \in C_r(G)} \rchi(\Pi^X)
\end{equation*}
of the reduced \Euc s of the subposets $\Pi^X$ as runs through the set
$C_r(G)$ of commuting $r$-tuples.

In this note we focus on equivariant \Euc s of partition posets.
Let $\Sigma_n$ denote the symmetric group of degree $n \geq 2$.  The set
$\Pi(\qt{\Sigma_{n-1}}{\Sigma_n})$ of partitions of the standard right
$\Sigma_n$-set $\qt{\Sigma_{n-1}}{\Sigma_n}$ is a (contractible) right
$\Sigma_n$-lattice with smallest element $\widehat 0$, the discrete partition,
and largest element $\widehat 1$, the indiscrete partition. We let
$\Pi^*(\qt{\Sigma_{n-1}}{\Sigma_n}) = \Pi(\qt{\Sigma_{n-1}}{\Sigma_n}) -
\{\widehat 0, \widehat 1\}$ be the (non-contractible) $\Sigma_n$-poset obtained
by removing $\widehat 0$ and $\widehat 1$. 


We now state the result and defer the the explanation of the undefined
expressions till after theorem.

\begin{thm}\label{thm:key}
  The $r$th reduced equvariant \Euc\ of the $\Sigma_n$-poset
  $\Pi^*(\qt{\Sigma_{n-1}}{\Sigma_n})$ of partitions is
\begin{equation*}
  \rchi_r(\Pi^*(\qt{\Sigma_{n-1}}{\Sigma_n}),\Sigma_n) =
  \frac{1}{n}(a \ast b_r)(n)
\end{equation*}
when $n \geq 2$ and $r \geq 1$.  
\end{thm}

The multiplicative arithmetic sequence $a \ast b_r$ is the Dirichlet
convolution
\begin{equation*}
  (a \ast b_r)(n) = \sum_{d_1d_2=n} a(d_1)b_r(d_2)
\end{equation*}
of the the multiplicative arithmetic sequences $a$ and $b_r$ given by
\begin{equation*}
a(n)=(-1)^{n+1}, \qquad
b_r(n) = \prod_p
(-1)^{n_p}p^{\binom{n_p}{2}}\binom{r}{n_p}_p, \qquad r \geq 1, n \geq 1  
\end{equation*}
where $n=\prod_p p^{n_p}$ is the prime factorization of $n$ and the
$p$-binomial coefficient
\begin{equation*}
  \binom{r}{d}_p = 
  \frac{(p^r-1)\cdots (p^r-p^{d-1})}{(p^{d}-1) \cdots (p^{d}-p^{d-1})}
\end{equation*}
is the number of $d$-dimensional subspaces of the $r$-dimensional
$\F_p$-vector space \cite[Proposition 1.3.18]{stanley97}.


\section{Partitions of finite $G$-sets}
\label{sec:white}

Let $G$ be a group and $S$ a finite right $G$-set.

\begin{defn}\label{defn:PiSG} 
\begin{enumerate}
\item A partition $\pi$ of $S$ is an equivalence relation on $S$.  The blocks
  of $\pi$ are the equivalence classes of $\pi$.  For any $x \in S$,
  $[x]_\pi$, or simply $[x]$, is the $\pi$-block of $x$. The set of
  $\pi$-blocks is denoted $\pi \backslash S$.
  
\item $\Pi(S)$ is the $G$-lattice of all partitions of $S$ and
    $\Pi^*(S)=\Pi(S) - \{\widehat 0, \widehat 1\}$ the $G$-poset of all
    partitions of $S$ but the discrete and the indiscrete partitions,
    $\widehat 0$ and $\widehat 1$.

  \item A partition of $S$ is a $G$-partition if $x \sim y \iff xg \sim yg$
    for all $x,y \in S$ and $g \in G$.

  \item $\Pi(S)^G$ is the lattice of all $G$-partitions of $S$ and $\Pi^*(S)^G
    = \Pi(G)^G - \{ \widehat 0, \widehat 1\}$ the poset of all $G$-partitions
    of $S$ but the discrete and indiscrete partitions.

  \item The isotropy subgroup at $x \in S$ is the subgroup $\isot xG = \{ g
    \in G \mid xg = x \}$ of $G$.

  \item If $\pi$ is a $G$-partition, the block isotropy subgroup at $x \in S$
    is the isotropy subgroup $\isot {[x]}G = \{ g \in G \mid (xg) \pi x \}$ at
    the $\pi$-block $[x]$ of $x$ in the $G$-set $\pi \backslash S$ of
    $\pi$-blocks.

\item The $G$-set $S$ is {\em isotypical\/} if all isotropy subgroups are
  conjugate. \label{defn:PiSG7}

\item The $G$-partition $\pi \in \Pi(S)^G$ is {\em isotypical\/} if the
  $G$-set $\qt{\pi}S$ of $\pi$-blocks is
  isotypical. $\Pi^{\mathrm{iso}}(S)^G$ is the poset of all isotypical
  $G$-partitions and $\Pi^{*+\mathrm{iso}}(S)^G = \Pi^{\mathrm{iso}}(S)^G - \{
  \widehat 0, \widehat 1\}$.
  \label{defn:PiSG8}

  \end{enumerate} 
\end{defn}

The set $\Pi(S)$ of partitions of $S$ is partially ordered by refinement:
\begin{equation*}
  \pi_1 \leq \pi_2 \iff \forall x \in S \colon
  [x]_{\pi_1} \subseteq [x]_{\pi_2}
\end{equation*}
The {\em meet\/} of $\pi_1$ and $\pi_2$ is the partition $\pi_1 \wedge \pi_2$
with blocks $[x]_{\pi_1 \wedge \pi_2} = [x]_{\pi_1} \cap [x]_{\pi_2}$, $x \in
S$.  The discrete partition is $\widehat 0$ with blocks $[x]_{\widehat 0}
=\{x\}$, $x \in S$, and the indiscrete partition is $\widehat 1$ with block
$[x]_{\widehat 1} = S$, $x \in S$.

\begin{exmp}\label{exmp:omegaK}
  Let $K$ be a subgroup of $G$. The partition $\omega_K$, whose blocks
  $[x]_{\omega_K}=xK$ are the $K$-orbits in $S$, is an $N_G(K)$-partition of
  $S$. In particular, the partition $\omega_G$ whose blocks are the $G$-orbits
  is a $G$-partition.
\end{exmp}

The set $\Pi(S)$ of partitions of $S$ is a right $G$-lattice: For any
partition $\pi$ of $S$ and any $g \in G$, $\pi g$ is the partition given by
$x(\pi g)y \iff (xg)\pi(yg)$. Then $[x]_{\pi g}g = \{yg \mid x(\pi g)y \} =
\{yg \mid (yg)\pi(xg) \} = \{ y \mid y\pi(xg) \} = [xg]_{\pi}$. Obviously,
\begin{equation*}
  \text{$\pi$ is a $G$-partition} \iff
  \forall g \in G \colon \pi g = \pi \iff
  \forall g \in G \forall x \in X \colon [x]_\pi g = [xg]_\pi \iff
  \forall g \in G \forall b \in \pi \colon bg \in \pi
\end{equation*}
Thus the fixed poset for this $G$-action on $\Pi(S)$, $\Pi(S)^G$, is the set
of all $G$-partitions.  The discrete and the indiscrete partitions are
$G$-partitions.



\begin{prop}\label{prop:isotropy}
  Let $\pi$ be a $G$-partition of $S$.
  \begin{enumerate}
  \item There is a right $G$-action on the set $\pi \backslash S$ of
    $\pi$-blocks such that $S \to \pi \backslash S$ is a $G$-map.
\item ${}_{x}G \leq {}_{[x]}G$ for any $x \in S$.
\item ${}_{xg}G = {}_{x}G^g$ and ${}_{[xg]}G = {}_{[x]}G^g$ 
\item  ${}_{xg}G \leq {}_{[x]}G^g$ for any $x \in S$ and any $g \in G$.
  \end{enumerate}
\end{prop}
\begin{proof}
  The $G$-action on $\pi \backslash S$ is given by $[x]g = [xg]$ for all $x
  \in S$ and $g \in G$.
\end{proof}

\begin{defn}\label{defn:contractor}
  Let $P$ be a subposet of a lattice. An element $c$ of $P$ is a contractor if
  $x \vee c \in P$ or $x \wedge c \in P$ for all $x \in P$.
\end{defn}

If $c$ is a contractor for $P$ then $x \leq x \vee c \geq c$ or $x \geq x
\wedge c \leq c$ are homotopies between the identity map of $P$ and the
constant map $c$. We view $P$ as a finite topological space with the order
right ideals as open sets.

\begin{lemma}\cite[Lemma 7.1]{arone:2015}\label{lemma:isotypical}
  $\Pi^*(S)^G$ is contractible unless $S$ is isotypical.
\end{lemma}
\begin{proof}
  Let $\omega_G$ be the $G$-partition represented by the $G$-map $S \to S/G$
  to the $G$-set of $G$-orbits and $\theta_G$ the $G$-partition represented by
  the $G$-map $S \to S/G \to \qt{\cong}{S/G}$ to the set of iso\m\ classes of
  $G$-orbits. Explicitly, $x \omega_G y$ if and only if $x$ and $y$ are in the
  same $G$-orbit, and $x \theta_G y$ if and only if $x$ and $y$ have conjugate
  isotropy subgroups. 
  We shall prove that $\theta_G$ is a contractor
  (Definition~\ref{defn:contractor}) for $\Pi^*(S)^G$ when $S$ is not
  isotypical.

  We first make some small observations.  Obviously, $\omega_G \leq
  \theta_G$. The $G$-action is trivial if and only if $\omega_G = \widehat
  0$. The $G$-action is isotypical if and only if $\theta_G= \widehat 1$.  If
  the $G$-action is trivial, all isotropy subgroups are equal to $G$, and
  therefore $\theta_G = \widehat 1$.  We may summarize these observation in a
  string
  \begin{equation*}
    \theta_G = \widehat 0 \implies \omega_G = \widehat 0 
    \iff \forall x \in S \colon {}_xG=G \implies \theta_G = \widehat 1 
    \iff \text{$S$ is isotypical}
  \end{equation*}
  of implications.

  Let $\pi$ be any $G$-partition of $S$. We claim that
  \begin{equation}\label{eq:claim1}
    \pi \wedge \theta_G = \widehat 0 \implies \pi = \widehat 0
  \end{equation}
  To see this first note that
  \begin{equation*}
   \forall x,y \in S \colon x \pi y \implies  y \cdot {}_xG 
   \subseteq [y]_{\pi \wedge \theta_G}
  \end{equation*}
  Indeed, let $x \pi y$ and $g \in {}_xG$. Then $y \pi (yg)$ for $y
  \pi x$, $x=xg$, and $(xg) \pi (yg)$. Thus $y$ and $yg$ are both in $[y]_\pi$
  and in $[y]_{\theta_G}$. Now assume that $\pi \wedge \theta_G = \widehat
  0$. Then
  \begin{equation*}
    \forall x,y \in S \colon x \pi y \implies {}_xG \leq {}_yG
  \end{equation*}
  for the block $[y]_{\pi \wedge \theta_G} = [y]_{\widehat 0} = \{y\}$ consists
  of $y$ alone which forces $yg = y$ for all $g \in {}_xG$. This can be
  sharpened to
  \begin{equation*}
    \forall x,y \in S \colon x \pi y \implies {}_xG = {}_yG
  \end{equation*}
  as the equivalence relation $\pi$ is symmetric, of course. Now, when $x$ and
  $y$ have the same isotropy subgroups, $x$ and $y$ belong to the same block
  under $\theta_G$. Thus we have shown $\pi \leq \theta_G$. Then $\pi = \pi
  \wedge \theta_G = \widehat 0$. This proves claim \eqref{eq:claim1}.

  Suppose that $S$ is not isotypical.  Then $\theta_G \neq \widehat 0,
  \widehat 1$ and $\theta_G$ belongs to the poset $\Pi^*(S)^G$. From claim
  \eqref{eq:claim1} we know that $\pi \wedge \theta_G \neq \widehat 0$ for all
  $\pi \in \Pi^*(S)^G$.  Thus $\theta_G$ is a contractor for $\Pi^*(S)^G$.
\end{proof}

There are, of course, isotypical $G$-sets $S$ for which $\Pi^*(S)^G$ is
contractible. 

\begin{exmp}[An isotypical $G$-set $S$ such that $\Pi^*(S)^G$ is contractible]
  Suppose that the Frattini subgroup $\Phi(G)$ of $G$ is nontrivial and
  proper. The $G$-set $S=G$ is transitive and hence isotypical. But still the
  poset $\Pi^*(S)^G$ is contractible: By Proposition~\ref{prop:GpiS},
  $\Pi^*(S)^G$ is the poset $(1,G)$ of non-identity proper subgroups of $G$,
  and $\Phi(G)$ is a contractor of $(1,G)$. (I thank Matthew Gelvin for
  pointing out this example.)
\end{exmp}

A $G$-partition of a {\em transitive\/} $G$-set $S$ is uniquely determined by
its block isotropy subgroup at a single point.

\begin{prop}\cite[Lemma 3]{white:76}\label{prop:GpiS}
  Let $S$ be a transitive $G$-set and $x$ a point of $S$. The block isotropy
  map
  \begin{equation*}
    \Pi(S)^G \to [{}_xG,G] = {}_xG/\cat SG{} \colon 
    \pi \to  {}_{[x]_\pi}G 
  \end{equation*}
  is an iso\m\ of posets.
\end{prop}
\begin{proof}
  Let $H={}_xG$ be the isotropy subgroup of $x$.  For every subgroup $K$ of
  $G$ containing $H$, let $\pi_K$ be the $G$-partition of $S$ with blocks
  $xKg$, $g \in G$ (the fibres of $S=\qt HG \to \qt KG$).
  The $\pi_K$-block of $x$, $[x]_{\pi_K} = xK$, has isotropy subgroup $\{g \in
  G \mid xg \in xK \}=K$. Conversely, let $\pi$ be any $G$-partition of
  $S$. The orbit through $x$ of the block isotropy subgroup ${}_{[x]_\pi}G$ is
  $x \cdot {}_{[x]_\pi}G = [x]_\pi$ as $S$ is transitive. These observations
  show that $K \to \pi_K$ is an inverse to the block isotropy subgroup map
  $\pi \to {}_{[x]_\pi}G$. It is clear that these bijections respect the
  partial orderings.
\end{proof}

\begin{defn}\label{defn:OG}
  $\cat OG{}$ is the category of finite $G$-sets with surjective $G$-maps as
  \m s. 
\end{defn}


We may consider $G$-partitions as \m s in the category $\cat OG{}$.  To any
$G$-partition $\pi$ of the $G$-set $S$ we associate the surjective $G$-map $S
\to \qt {\pi}S$. Conversely, the blocks of the partition represented by the
surjective $G$-map $\pi \colon S \to T$ are the fibres of $\pi$. The block of
$x \in S$ is $\pi^{-1}(\pi(x))$. The overlap of the block and the $G$-orbit of
$x$ is the orbit through $x$ of the block isotropy subgroup, $\pi^{-1}(\pi(x))
\cap xG = x \isot {\pi(x)}G$.

\section{Euler characteristics of posets of $G$-partitions}
\label{sec:eucII}

Let $\Pi$ be a finite poset. For $a,b \in \Pi$ let
\begin{center}
\begin{tabular}{*{3}{>{$}l<{$}}} 
  a/\Pi = \{p \in \Pi \mid a \leq p \} \qquad {}  
  & a//\Pi = \{p \in \Pi \mid a<p \} \qquad {}  &
  k^a = -\rchi(a//\Pi) \\
  \Pi/b = \{ p \in \Pi \mid p \leq b\} \qquad {}   & 
  \Pi//b = \{ p \in \Pi \mid p<b\} \qquad {}   &
   k_b=-\rchi(\Pi//b)
\end{tabular}
\end{center}
denote the coslice of $\Pi$ under $a$, the proper coslice of $\Pi$ under $a$,
and the \we\ at $a$, and, dually, the slice of $\Pi$ over $b$, the proper
slice of $\Pi$ over $b$, and the co\we\ at $b$ \cite[Corollary 3.8]{gm:2012}.
The \Euc\ of $\Pi$
\begin{equation*}
  \sum_{a \in \Pi} k^a = \chi(\Pi) = \sum_{b \in \Pi} k_b
\end{equation*}
is the sum of the values of the \we\ or co\we . In particular, for a finite
$G$-set $S$, we can compute the \Euc\ of $\Pi^*(S)^G$,
\begin{equation}\label{eq:gm}
  \sum_{\pi \in \Pi^*(S)^G} -\rchi(\pi//\Pi^*(S)^G) = \chi(\Pi^*(S)^G) 
   = \sum_{\pi \in \Pi^*(S)^G} -\rchi(\Pi^*(S)^G//\pi)
\end{equation}
from its \we\ or co\we\ \cite[Corollary 3.8]{gm:2012}. We shall now determine
these functions.

\begin{prop}[Slices in $\Pi^*(S)^G$]\label{prop:we}
  For any $G$-partition $\pi$ of the right $G$-set $S$
  \begin{equation*}
    \pi / \Pi(S)^G = \Pi(\pi \backslash S)^G, \qquad
    \pi // \Pi^*(S)^G = \Pi^*(\pi \backslash S)^G
  \end{equation*}
 The \we\ for $\Pi^*(S)^G$
 \begin{equation*}
   k^\pi = -\rchi(\Pi^*(\pi \backslash S)^G), \qquad \pi \in \Pi^*(S)^G,
 \end{equation*}
 vanishes at $\pi$ unless $\pi$ is isotypical
 (Definition~\ref{defn:PiSG}.\eqref{defn:PiSG8}). 
\end{prop}
\begin{proof}
  Let $\rho$ be a partition of the right $G$-set $\pi \backslash G$ of
  blocks of $\pi$. There is then a partition of $S$ with blocks $[x] =
  [[x]_\pi]_\rho$, $x \in S$. This new partition is a $G$-partition if and only
  if $\rho$ is a $G$-partition of $\pi \backslash S$. Any $G$-partition $\geq
  \pi$ of $S$ arises in this way. 
\end{proof}

\begin{prop}[Coslices in $\Pi^*(S)^G$]\label{prop:cowe}
 For any $G$-partition $\pi$ of the right $G$-set $S$
\begin{equation*}
 \Pi(S)^G/\pi = \prod_{BG \in \pi \backslash S /G } \Pi(B)^{_{B}G}, \qquad
 \Pi^*(S)^G//\pi =
 \big(\prod_{BG \in \pi \backslash S /G} \Pi(B)^{_{B}G}\big)^* 
\end{equation*}
The co\we\ for $\Pi^*(S)^G$
\begin{equation*}
  k_\pi = - \prod_{\substack{BG \in \pi \backslash S /G \\ |B|>1}}
  \rchi(\Pi^*(B)^{{}_BG}),\qquad \pi \in \Pi^*(S)^G,
\end{equation*}
vanishes at $\pi$ unless all blocks $B$ of $\pi$ are isotypical $_BG$-sets.
\end{prop}
\begin{proof}
  Let $\pi$ be a $G$-partition and $B$ one its blocks.  Observe first that the
  blocks contained in $B$ of a $G$-partition $\lambda \leq \pi$ determine all
  blocks of $\lambda$ contained in any of the blocks of the orbit $BG$ through
  $B$ for the $G$-action on $\pi \backslash S$.
 
  Let $B$ be a block, with isotropy subgroup $_BG$, of the $G$-partition
  $\pi$. Let $\lambda$ be a $_BG$ partition of $B$. Extend $\lambda$ to a
  $G$-partition of the orbit $BG$ of $B$ in $\pi$ by $[xg]_\lambda=[x]_\lambda
  g$. We must argue that this extension is well-defined. Suppose that $x_1g_1
  = x_2g_2$ for some $x_1,x_2 \in B$ and $g_1,g_2 \in G$. We must show that
  $[x_1]_\lambda g_1 = [x_2]_\lambda g_2$. We have $x_2 = x_2g_2g_2^{-1} =
  x_1g_1g_2^{-1}$. From $B=[x_2]_\pi = [x_1g_1g_2^{-1}]_\pi = [x_1]_\pi
  g_1g_2^{-1} = Bg_1g_2^{-1}$ we get that $g_1g_2^{-1}$ stabilizes the block
  $B$. As $\lambda$ is a $_BG$-partition, $[x_1]_\lambda g_1 = [x_1]_\lambda
  g_1g_2^{-1}g_2 = [x_1g_1g_2^{-1}]_\lambda g_2 = [x_2]_\lambda g_2$ as we
  wanted.

  Conversely, if $\lambda$ is a $G$-partition and $\lambda \leq \pi$ then the
  blocks of $\lambda$ inside a fixed block $B$ of $\pi$ form a $_BG$-partition
  of $B$, of course.  

  According to Quillen the reduced \Euc\ is multiplicative: $\rchi((\prod
  L_i)^*) = \prod \rchi(L_i^*)$ for lattices $L_i$ of more than one element
  \cite[Proposition 2.8]{arone:2015}.

  If the block $B$ of partition $\pi$ consists of a single element of $S$,
  then also the partition poset $\Pi(B)$ consists of a single element so it
  can be omitted from the poset product $\prod_{B \in \pi \backslash S}\Pi(B)$.
\end{proof}

In all cases, 
\begin{equation}\label{eq:wecowe}
  \sum_{\pi \in \Pi^*(S)^G} \rchi(\Pi^*(\pi \backslash S)^G) =
  -\chi(\Pi^*(S)^G) =
  \sum_{\pi \in \Pi^*(S)^G} 
   \prod_{\substack{BG \in \pi \backslash S /G \\ |B|>1}}
  \rchi(\Pi^*(B)^{{}_BG}) 
\end{equation}
where the sum on the left can be restricted to the $G$-partitions $\pi$ with
$G$-isotypical block set $\pi \backslash S$, and the sum on the right can be
restricted to the $G$-partitions $\pi$ for which $_BG$ acts isotypically on
every block $B$ of $\pi$. If $G$ acts non-isotypically on $S$ then these sums
equal $0$.

\begin{exmp}[Two examples of $G$-partition posets]
The poset $\Pi^*(S)^G$ of nontrivial $G$-partitions for $S=\{1,2,\ldots,4\}$ and
$G=\gen{(1,2)(4,5)} \leq \Sigma_4$ (isotypical):
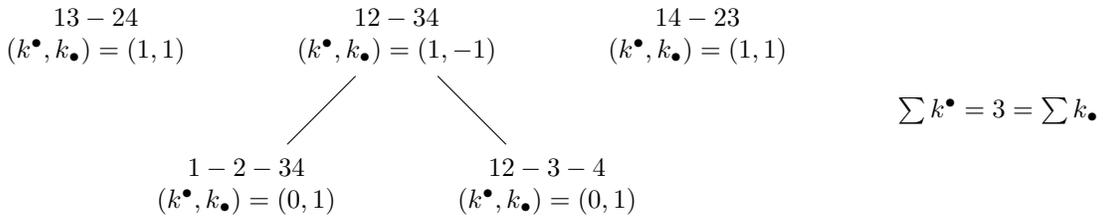
\begin{center}
  \begin{tikzpicture}
    \node (1) at (-4,2) {
      \begin{tabular}{cc}
        $13-24$ \\ $(k^\bullet, k_\bullet)=(1,1)$ 
      \end{tabular}};
    \node (3) at (0,2) {
      \begin{tabular}{cc}
        $12-34$ \\ $(k^\bullet, k_\bullet)=(1,-1)$ 
      \end{tabular}};
    \node (2) at (4,2) {
      \begin{tabular}{cc}
        $14-23$ \\ $(k^\bullet, k_\bullet)=(1,1)$ 
      \end{tabular}};
    \node (5) at (-2,0) {
      \begin{tabular}{cc}
        $1-2-34$ \\ $(k^\bullet, k_\bullet)=(0,1)$ 
      \end{tabular}};
    \node (4) at (2,0) {
      \begin{tabular}{cc}
        $12-3-4$ \\ $(k^\bullet, k_\bullet)=(0,1)$ 
      \end{tabular}};
    \draw[-] (5) -- (3);
    \draw[-] (4) -- (3);
    \node at (8,1) {$\sum k^\bullet = 3 = \sum k_\bullet$};
  \end{tikzpicture}
\end{center}

The poset $\Pi^*(S)^G$ of nontrivial $G$-partitions for $S=\{1,2,\ldots,6\}$ and
$G=\gen{(1,2,3),(4,5)} \leq \Sigma_6$ (non-isotypical):  
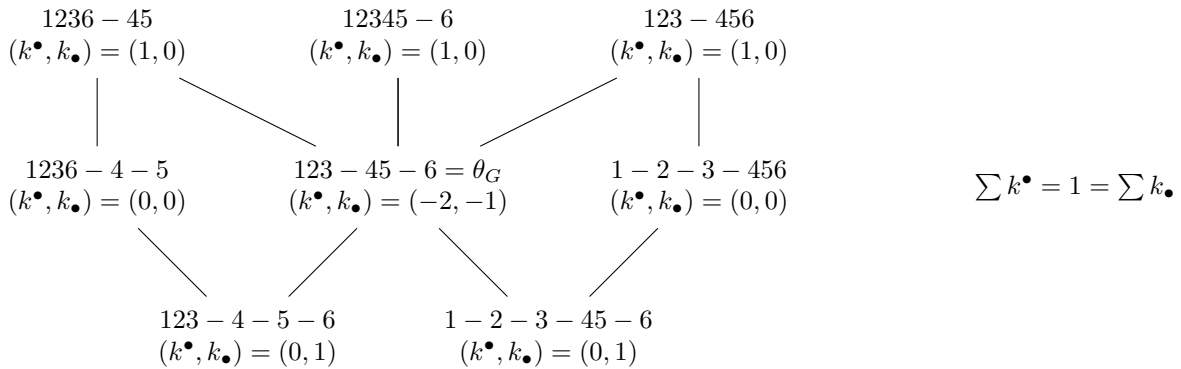
\begin{center}
\begin{tikzpicture}
   \node (1) at (-4,2) {
     \begin{tabular}{cc}
       $1236-45$ \\ $(k^\bullet, k_\bullet)=(1,0)$
     \end{tabular}};
   \node (3) at (0,2) {
     \begin{tabular}{cc}
       $12345-6$ \\ $(k^\bullet, k_\bullet)=(1,0)$ 
     \end{tabular}};
   \node (2) at (4,2) {
     \begin{tabular}{cc}
       $123-456$ \\ $(k^\bullet, k_\bullet)=(1,0)$ 
     \end{tabular}};
   \node (4) at (-4,0) {
     \begin{tabular}{cc}
       $1236-4-5$ \\ $(k^\bullet, k_\bullet)=(0,0)$ 
     \end{tabular}};
   \node (5) at (0,0) {
     \begin{tabular}{cc}
       $123-45-6 = \theta_G$ \\ $(k^\bullet, k_\bullet)=(-2,-1)$ 
     \end{tabular}};
   \node (7) at (4,0) {
     \begin{tabular}{cc}
       $1-2-3-456$ \\ $(k^\bullet, k_\bullet)=(0,0)$ 
     \end{tabular}};
   \node (6) at (-2,-2) {
     \begin{tabular}{cc}
       $123-4-5-6$ \\ $(k^\bullet, k_\bullet)=(0,1)$ 
     \end{tabular}};
   \node (8) at (2,-2) {
     \begin{tabular}{cc}
       $1-2-3-45-6$ \\ $(k^\bullet, k_\bullet)=(0,1)$ 
  \end{tabular}};
   \draw[-] (4) -- (1);
   \draw[-] (5) -- (3);
   \draw[-] (7) -- (2);
   \draw[-] (6) -- (4);
   \draw[-] (6) -- (5);
   \draw[-] (8) -- (5);
   \draw[-] (8) -- (7);
   \draw[-] (5) -- (1);
   \draw[-] (5) -- (2);
   \draw[-] (5) -- (3);
   \node at (9,0) {$\sum k^\bullet = 1 = \sum k_\bullet$};
\end{tikzpicture}
\end{center}
\end{exmp}

\begin{cor}\label{cor:iso}
  The inclusion $\Pi^{*+\mathrm{iso}}(S)^G \hookrightarrow \Pi^*(S)^G$ is a
  homotopy equivalence.
\end{cor}
\begin{proof}
  This follows immediately from Bouc's theorem \cite{bouc84a} since $\pi //
  \Pi^*(S)^G$ is contractible unless $\pi$ is isotypical by
  Proposition~\ref{prop:we} and Lemma~\ref{lemma:isotypical}.
\end{proof}

Because of Corollary~\ref{cor:iso} we now restrict attention to isotypical
$G$-partitions of isotypical $G$-sets.

For any $G$-orbit $S$ and any natural number $n \geq 1$, let 
$nS = \coprod_nS$ be the isotypical $G$-set with $n$ $G$-orbits isomorphic to
$S$. 

\begin{defn}\label{defn:GStirling}
Let $S$ and $T$ be $G$-orbits.
\begin{itemize}
\item An $nS/kT$-partition is an isotypical $G$-partition of $nS$ with block
  $G$-set isomorphic to $kT$.
\item The $G$-Stirling number of the second kind 
\begin{equation*}
  S_G(nS,kT) = |\{ \pi \in \Pi(nS)^G \mid \qt{\pi}{(nS)} \cong kT \}|
\end{equation*}
is the number $nS/kT$-partitions.
\end{itemize}
\end{defn}

In the following, $\cat SG{}$ is the poset of subgroups, and $[\cat SG{}]$ the
poset of subgroup conjugacy classes of $G$.  We write $\zeta_G$, or just
$\zeta$, for the poset incidence matrix (with $\zeta_G(H,K)=1$ if $H \leq K$
and $\zeta_G(H,K)=0$ otherwise) and $\mu=\mu_G = \zeta_G^{-1}$ for the \Mb\
matrix of $\cat SG{}$.

\begin{defn}\label{defn:StirlingMatrix}
  The $G$-Stirling matrix of degree $n$ is the square $(n|[\cat SG{}]| \times
  n|[\cat SG{}]|)$-matrix
  \begin{equation*}
    [\zeta]_G \otimes S_G =
\big((S_G(s\qt HG,t\qt KG))_{ 1 \leq s,t \leq n}\big)_{H,K \in [\cat SG{}]}    
  \end{equation*}
   obtained as
  the $(|[\cat SG{}]| \times |[\cat SG{}]|)$-matrix of $(n \times n)$-block
  matrices $(S_G(s\qt HG,t\qt KG))_{1 \leq s,t \leq n}$ of Stirling numbers
  with fixed $G$-orbits $\qt GH$ and $\qt KG$.
\end{defn}

If we order the subgroups of $G$ in decreasing order starting with $G$ itself,
the $G$-Stirling matrix is lower triangular.

If we in Equation~\ref{eq:gm} insert the values from Proposition~\ref{prop:we}
we obtain formulas for the reduced \Euc\ of the poset $\Pi^*(nS)^G$,
\begin{equation}\label{eq:we}
  \rchi(\Pi^*(nS)^G) = -1 - \sum_{T,k} \rchi(\Pi^*(kT)^G) S_G(nS,kT), \qquad
  1 =  \sum_{k|T|>1} -\rchi(\Pi^*(kT)^G) S_G(nS,kT) 
\end{equation}
with $T$ ranging over the set of iso\m\ classes of $G$-orbits and $k \geq 1$
over natural numbers with $k|T|>1$.  (Observe that $S_G(nS,nS) = 1$.)  In
matrix notation
\begin{equation}\label{eq:wematrix}
  \begin{pmatrix}
    S_G(s\qt HG,t\qt KG)
  \end{pmatrix}_{\substack{H,K \in [\cat SG{}] \\ 1 \leq s,t \leq n}}
  \begin{pmatrix}
    \vdots \\ -\rchi(\Pi^*(s\qt HG)^G) \\ \vdots
  \end{pmatrix}_{\substack{S \in [\cat SG{}] \\ 1 \leq s \leq n}} =
  \begin{pmatrix}
   0 \\ 1 \\ \vdots \\ 1
  \end{pmatrix}
\end{equation}
we see that minus the reduced \Euc s of the $G$-partitions of the isotypical
$G$-sets are a \we\ for the Stirling matrix of $G$.
Equation~\eqref{eq:wematrix} comes with the caveat that the top entry of the
left column vector is $0$ and not $-\rchi(\Pi^*(1\qt GG)^G)=1$.

\begin{exmp}[$G$-Stirling matrices of degree $1$]\label{exmp:Stmatrixdeg1}
The Stirling number for single orbits $S=\qt HG$ and $T=\qt KG$,
\begin{equation*}
  S_G(\qt HG,\qt KG) =
    \cat SG{}(H,[K])
  = \frac{|N_G(H,K)|}{|N_G(K,K)|}
  = \frac{|\cat SG{}(\qt HG,\qt KG)|}{|\cat SG{}(\qt KG)|}
  = \frac{\left| (\qt KG)^H \right|}{\left| (\qt KG)^K \right|}
  = \frac{\mathrm{TOM}(H,K)}{\mathrm{TOM}(K,K)}
\end{equation*}
is the number, $\cat SG{}(H,[K]) = |\{ L \in [K] \mid H \leq L \}|$, of
conjugates of $K$ containing $H$ \cite[Definition 3.5, Lemma
3.6]{jmm:eulercent}. This number is determined by the table of marks
$\mathrm{TOM}(H,K) = |(\qt KG)^H|$ for $G$.  Proposition~\ref{prop:GpiS} or
\cite{jmm:eulercent} show that the entries of the column vector in
Equation~\eqref{eq:wematrix} are
\begin{equation*}
  -\rchi(\Pi^*(\qt HG)^G) = -\rchi(H,G) = -\mu(H,G)
\end{equation*}
for all proper subgroups $H$ of $G$. (In any finite poset, $\mu(x,y) =
\rchi(x,y)$ whenever $x < y$ \cite[Proposition 3.8.5]{stanley97}.)

For instance, $G=\Sigma_3$ has $|[\cat S{\Sigma_3}{}]|=4$ orbits
$S_1,S_2,S_3,S_6$ of sizes $1,2,3,6$.  The $\Sigma_3$-Stirling matrix of
degree $1$ is
\begin{center}
  \begin{tabular}[t]{*{1}{>{$}c<{$}} | *{4}{>{$}c<{$}} |  *{1}{>{$}c<{$}} }
      S_{\Sigma_3}(S,T) & \qt {\Sigma_3}{\Sigma_3} &
   \qt {A_3}{\Sigma_3} & \qt {C_2}{\Sigma_3} & \qt {C_1}{\Sigma_3} 
   & -\rchi(\Pi^*(\qt H{\Sigma_3})^{\Sigma_3}) \\ \hline
   \qt {\Sigma_3}{\Sigma_3} & 1 & {} & {} & {} & 0 \\
   \qt {A_3}{\Sigma_3} & 1 & 1 & {} & {} & 1 \\
   \qt {C_2}{\Sigma_3} & 1&  0 & 1 & {} & 1 \\
   \qt{C_1}{\Sigma_3} & 1& 1 & 3 & 1 & -3 
  \end{tabular}
\end{center}
and (remembering the caveat that the top entry of the column to the far right
is $0$ when solving Equation~\eqref{eq:wematrix}) we read off that
$\mu(A_3,\Sigma_3)=-1$, $\mu(C_2,\Sigma_3)=-1$, $\mu(1,\Sigma_3)=3$.
\end{exmp}

Since $\rchi(\Pi^*(1 \qt HG)^G) = \rchi(H,G) = \mu(H,G)$ for proper subgroups
$H$ of $G$ by Proposition~\ref{prop:GpiS}, it seems natural to define the
higher \Mb\ numbers to be the solutions to the linear
equation~\eqref{eq:wematrix}.

\begin{defn}[Higher \Mb\ numbers] For every subgroup $H$ of $G$ and every
natural number $n \geq 1$ let
  \begin{equation*} \mu_n(H,G) = \rchi(\Pi^*(n\qt HG)^G)
  \end{equation*} with the convention that $\mu_1(G,G)=1$.
\end{defn}

For any group $G$, $\mu_n(G,G) = (-1)^n(n-1)! = \mu_n(1,1)$ for $n \geq 2$,
and $\mu_n(1,G) = \rchi(\Pi^*(\coprod_nG)^G)$ for $n \geq 1$. With $n=1$,
$\mu_1(H,G) = \mu(H,G)$ is the usual \Mb\ function of $\cat SG{}$ as
considered in Example~\ref{exmp:Stmatrixdeg1}.

The higher \Mb\ numbers $\mu_h(H,G)$ for $1 \leq h \leq n$ are determined by the
$G$-Stirling matrix of degree $n$. We shall now consider the problem of
determining the entries of this matrix.

Let $S(n,k)$ stand both for the poset of partitions of the $n$ element set
with $k$ blocks and for the Stirling number (Example~\ref{exmp:G=1}) of such
partitions. Then
\begin{equation*}
  S_G(n \qt HG, k\qt KG) = \sum_{\pi \in S(n,k)} \prod_{b \in \pi} 
  \frac{|\cat OG{}(\qt HG, \qt KG)|^{|b|}}{|\cat OG{}(\qt KG),\qt KG|} =
  \frac{|\cat OG{}(\qt HG, \qt KG)|^n}{|\cat OG{}(\qt KG)|^k} S(n,k) =
   \frac{\mathrm{TOM}(H,K)^n}{\mathrm{TOM}(K,K)^k} S(n,k)
\end{equation*}
In particular
\begin{equation}\label{eq:SGGab}
   S_G(n S, kT)=
  \begin{cases}
    |T|^{n-k} S(n,k) & \cat OG{}(S,T) \neq  \emptyset \\ 
    0 & \cat OG{}(S,T) =  \emptyset 
  \end{cases}
\end{equation}
when $G$ is abelian.

\begin{lemma}\label{lemma:munormal}
  If $H \trianglelefteq G$ is normal in $G$, then $\mu_n(H,G) = \mu_n(1,\qt
  HG)$ for all $n \geq 1$.
\end{lemma}
\begin{proof}
  $H$ acts trivially on $\qt HG$ as $Hgh = Hghg^{-1}g=Hg$ for all $h \in H$,
  $g \in G$. Thus a partition of $n \qt HG$ is a $G$-partition if and only if
  it is a $\qt HG$-partition.
\end{proof}

The higher \Mb\ numbers $\mu_1(H,G), \ldots, \mu_n(H,G)$ for $H \leq G$
(except for $\mu_1(G,G)$ which by decree equals $1$) solve the system
of linear equations \eqref{eq:wematrix} which we now rewrite as
\begin{equation}\label{eq:linsyst}
[\zeta]_G \otimes S_G
  \begin{bmatrix}
    0 \\ -\mu_2(G,G) \\ \vdots \\ -\mu_n(G,G) \\ \vdots \\
    -\mu_1(H,G) \\ \vdots \\ -\mu_n(H,G) \\ \vdots \\
    -\mu_1(1,G) \\ \vdots \\ -\mu_n(1,G)
  \end{bmatrix} =
\begin{bmatrix}
  0 \\ 1 \\ \vdots \\ 1 \\ \vdots \\ 1 \\ \vdots \\ 1 \\ \vdots \\ 1 \\ \vdots
  \\ 1
\end{bmatrix} 
\end{equation}
with the $G$-Stirling matrix as coefficient matrix. We shall adapt the
convention that in the Stirling matrix the groups will be listed with
decreasing order. The group $G$ itself occurs as the first group in the
Stirling matrix which is lower triangular. The first $n$ columns are made up
of the block matrices $(S(i,j)_{1 \leq i,j \leq n}$ of classical Stirling
numbers. All entries of the first column, in particular, equal
$S(n,1)=1$. Thus
\begin{equation*}
  [\zeta]_G \otimes S_G
  \begin{bmatrix}
    \mu_1(G,G) \\ \vdots \\ \mu_n(G,G) \\ \vdots \\ \mu_1(1,G) \\ \vdots \\
    \mu_n(1,G)  
  \end{bmatrix} = 
  \begin{bmatrix}
    1 \\ 1 \\ \vdots \\ 1
  \end{bmatrix} -
  \begin{bmatrix}
    0 \\ 1 \\ \vdots \\ 1
  \end{bmatrix} =
  \begin{bmatrix}
    1 \\ 0 \\ \vdots \\ 0
  \end{bmatrix} 
\end{equation*}
or
\begin{equation}\label{eq:Stirlingfirst}
    \begin{bmatrix}
    \mu_1(G,G) \\ \vdots \\ \mu_n(G,G) \\ \vdots \\ \mu_1(1,G) \\ \vdots \\
    \mu_n(1,G)  
  \end{bmatrix} =
   ([\zeta]_G \otimes S_G)^{-1} 
  \begin{bmatrix}
    1 \\ 0 \\ \vdots \\ 0
  \end{bmatrix}
\end{equation}
The entries of the inverted matrix $ ([\zeta]_G \otimes S_G)^{-1}$ are the
{\em $G$-Stirling numbers of the first kind} \cite[p 36]{stanley97}.

\begin{exmp}[Higher \Mb\ numbers of the trivial group]\label{exmp:G=1}
  The $C_1$-Stirling matrix of the second kind
  (Definition~\ref{defn:GStirling}) is the matrix
 \begin{equation*}
   S=
    \begin{bmatrix}
      1 & {} & {} & {} & {} & {} \\
      1 & 1 & {} & {} & {} & {}  \\
      1 & 3 & 1 & {} & {} & {} \\
      1 & 7 & 6 & 1 & {} & {}  \\
      1 & 15 & 25 & 10 & 1 & {} \\
      1 & 31 & 90 & 65 & 15 & 1 
  \end{bmatrix}
 \end{equation*}
 of classical Stirling numbers $ S(n,k) = | \{ \pi \in \Pi_n \mid |\pi| = k
 \}$ of the second kind. The higher \Mb\ numbers of the trivial group are by
 Equation~\eqref{eq:Stirlingfirst} equal to the Stirling numbers of the first
 kind \cite[p 36]{stanley97}
\begin{equation*}
  \mu_n(1,1) = (S^{-1})(n,1) = s(n,1)=(-1)^{n-1} (n-1)!, \qquad n \geq 1
\end{equation*}
We have re-derived the classical formula \cite[Example 3.10.4]{stanley97} for
the reduced \Euc\ of the partition poset.
\end{exmp}

\begin{lemma}\label{lemma:muabelian}
  If the group $G$ is abelian then
  \begin{equation*}
    \mu_n(H,G) = \mu(H,G)|G:H|^{n-1}\mu_n(1,1)
  \end{equation*}
  for all $n \geq 1$ and all subgroups $H \leq G$.
\end{lemma}
\begin{proof}
   Since $G$ is abelian, $S_G(i\qt HG,j\qt KG) = |G:K|^{i-j}S(i,j)$ by
  Equation~\eqref{eq:SGGab}, and the $G$-Stirling matrix of degree $n$ is the
  block matrix
  \begin{equation*}
    \left(
      (\zeta(H,K) |G:K]^{i-j} S(i,j))_{1 \leq i,j \leq n} \right)_{H,K \in [\cat SG{}]}
  \end{equation*}
  The vector $\big((\mu_i(H,G))_{1 \leq i \leq n}\big)_{H \in \cat SG{}}$ is
  (Equation~\eqref{eq:Stirlingfirst}) the first column
  \begin{equation*}
    \big([\mu](H,K) |G:H]^{i-1} S^{-1}(i,1))_{1 \leq i \leq n}\big)_{H \in
      [\cat SG{}]} = 
     \big(\mu(H,K) |G:H]^{i-1} \mu_i(1,1))_{1 \leq i \leq n}\big)_{H \in
      [\cat SG{}]}
  \end{equation*}
  in the inverse matrix 
  \begin{equation*}
    \left(
(\mu(H,K) |G:H]^{i-j} S^{-1}(i,j))_{1 \leq i,j \leq n} \right)_{H,K \in [\cat SG{}]}
  \end{equation*}
 of the $G$-Stirling matrix.
\end{proof}

In the example below we consider an example of a Stirling matrix for a
non-abelian group.

\begin{exmp}
The $\Sigma_3$-Stirling matrix of degree $3$ (reusing the notation
of Example~\ref{exmp:Stmatrixdeg1}) is
\begin{center}
   \begin{tabular}[t]{*{1}{>{$}c<{$}} |  *{3}{>{$}c<{$}} |
   *{3}{>{$}c<{$}} | *{3}{>{$}c<{$}} | *{3}{>{$}c<{$}} |  *{1}{>{$}c<{$}} }
S_{\Sigma_3}(S,T) & 1S_1 & 2S_1 & 3S_1 & 
1S_2 & 2S_2 & 3S_2 & 1S_3&2S_3&3S_3 &1S_6&2S_6&3S_6 
  &  -\mu_i(H,\Sigma_3) \\ \hline 
1S_1 &  1  & 0  & 0  & {}  & {}  & {}  & {}  & {}  & {}  & {}  & {}  & {} & 0 \\
2S_1 &  1  & 1  & 0  & {}  & {}  & {}  & {}  & {}  & {}  & {}  & {}  & {} &1 \\
3S_1 &  1  & 3  & 1  & {}  & {}  & {}  & {}  & {}  & {}  & {}  & {}  & {} & -2\\ \hline
1S_2 &  1  & 0  & 0  & 1  & 0  & 0  & {}  & {}  & {}  & {}  & {}  & {}& 1\\ 
2S_2 &  1  & 1  & 0  & 2  & 1  & 0  & {}  & {}  & {}  & {}  & {}  & {} & -2\\
3S_2 &  1  & 3  & 1  & 4  & 6  & 1  & {}  & {}  & {}  & {}  & {}  & {} & 8\\ \hline
1S_3 &  1  & 0  & 0  & {}  & {}  & {}  & 1  & 0  & 0  & {}  & {}  & {} & 1\\
2S_3 &  1  & 1  & 0  & {}  & {}  & {}  & 1  & 1  & 0  & {}  & {}  & {} & -1\\ 
3S_3 &  1  & 3  & 1  & {}  & {}  & {}  & 1  & 3  & 1  & {}  & {}  & {} & 2\\ \hline
1S_6 &  1  & 0  & 0  & 1  & 0  & 0  & 3  & 0  & 0  & 1  & 0  & 0 & -3\\
2S_6 &  1  & 1  & 0  & 2  & 1  & 0  & 9  & 9  & 0  & 6  & 1  & 0 & 18 \\
3S_6 &  1  & 3  & 1  & 4  & 6  & 1  &27  &81  &27  &36  &18  & 1 & -216
\end{tabular}
\end{center}
We read off that $\mu_n(A_3,\Sigma_3) = \mu_n(1,C_2) = -2^{n-1}\mu_n(1,1)$
(Lemma~\ref{lemma:munormal}) and that $\mu_n(1,\Sigma_3) =
-3^n\mu_n(1,1)$. This last result shows that Lemma~\ref{lemma:muabelian} does
not in general extend to non-abelian groups.
\end{exmp}

\section{Equivariant \Euc s of $G$-posets}
\label{sec:eucs}

Let $\Pi$ be a finite $G$-poset. The $r$th, $r \geq 1$, equivariant \Euc\ of
$\Pi$ is \cite{atiyah&segal89} \cite[Proposition 2.9]{jmm:eulercent}
\begin{equation*}
  \chi_r(\Pi,G) = \frac{1}{|G|} \sum_{X \in C_r(G)}\chi(\Pi^X) 
  = \frac{1}{|G|} \sum_{A \leq G}\chi(\Pi^A)\varphi_r(A) 
\end{equation*}
The first sum runs over the set $C_r(G)$ of all commuting $r$-tuples
$X=(x_1,\ldots,x_r)$ of elements of $G$. The second sum runs over all abelian
subgroups $A$ of $G$ and $\varphi_r(A)$ is the number of generating $r$-tuples
$(a_1,\ldots,a_r)$ of elements of $A$ \cite{hall36} \cite[Remark
2.20]{jmm:eulercent}.

We now specialize from general poset to posets of partitions.
Let $S$ be a finite $G$-set, $\Pi(S)$ the $G$-poset of partitions of $G$, and 
 $\Pi^*(S)=\Pi(S)-\{\widehat 0, \widehat 1 \}$ the $G$-poset of non-extreme
 partitions of $S$.

\begin{defn}
  The group $G$ acts effectively on $S$ if only the trivial element of $G$
  fixes all elements of $S$.
\end{defn}

\begin{lemma}\label{lemma:Aisotypical}
Suppose that the abelian group $A$ acts effectively on $S$.
The following conditions are equivalent:
\begin{enumerate}
\item $A$ acts isotypically on $S$ \label{lemma:Aisotypical1}
\item $A$ acts freely on $S$ \label{lemma:Aisotypical2} 
\item The degree of any non-identity element of $A$ is
  $|S|$ \label{lemma:Aisotypical3} 
\item The cycle structure of any element of $A$ is $d^m$ for some natural
  numbers $d$ and $m$ with $dm=|S|$ \label{lemma:Aisotypical4}
\end{enumerate}
\end{lemma}
If $A$ acts isotropically on $S$ then the order of $A$ divides $|S|$.
\begin{proof}
  If $A$ acts isotypically and $A$ is abelian, the isotropy subgroup at any
  point of $S$ is the same subgroup, $B$, of $A$. The group $B$ acts trivially
  on $S$, so $B$ is the trivial subgroup since the action is effective. Thus
  $A$ acts freely on $S$.  

  If $A$ acts isotropically on $S$ then $S=m \qt 1A$ as right $A$-sets and
  $|S|=m|A|$.
\end{proof}

\begin{lemma}\label{lemma:key}
   Let $A$ be any abelian subgroup of $\Sigma_n$ acting freely on
$\qt{\Sigma_{n-1}}{\Sigma_n}$. Put $m=\frac{n}{|A|}$.
\begin{enumerate}
\item \label{lemma:key1} The number of conjugates of $A$ in $\Sigma_n$ is
  \begin{equation*}
  |\Sigma_n : N_{\Sigma_n}(A)| = \frac{1}{|\Aut{}A|} \frac{n!}{|A|^mm!}
  \end{equation*}
\item
  $\rchi(\Pi^*(\qt{\Sigma_{n-1}}{\Sigma_n})^A)=(-1)^{m-1}\mu(1,A)|A|^{m-1}(m-1)!$
  when $n \geq 2$. \label{lemma:key2}
\item \label{lemma:key3} $\rchi(\Pi^*(\qt{\Sigma_{n-1}}{\Sigma_n})^A)
  |\Sigma_n : N_{\Sigma_n}(A)| = -(-1)^{n/|A|}\mu(1,A)\frac{1}{|\Aut{}A|}(n-1)!$
\end{enumerate}
\end{lemma}
\begin{proof}
  \eqref{lemma:key1} It is a standard result that the normalizer of $A$ in the
  right regular permutation representation of $A$ is the holomorph $A \rtimes
  \Aut{}A$ of $A$ \cite[pp 36--37]{robinson:groups}. Similarly, the normalizer
  of $A$ in $m$ times the right regular representation is $(A \wr \Sigma_m)
  \rtimes \Aut{}A$ of order $|\Aut{}A||A|^mm!$.

  \noindent \eqref{lemma:key2} As an $A$-set $\qt{\Sigma_{n-1}}{\Sigma_n}= m \qt 1A$
  consists of $m$ free $A$-orbits. According to Lemma~\ref{lemma:muabelian}
  \begin{equation*}
    \rchi(\Pi^*(\qt{\Sigma_{n-1}}{\Sigma_n})^A) = \rchi(\Pi^*(m\qt 1A)^A) =
    \mu(1,A)|A|^{m-1}\mu_m(1,1) = (-1)^{m-1}\mu(1,A)|A|^{m-1}(m-1)!
  \end{equation*}
  This formula also holds when $A$ is trivial group. In this case, the left
  hand side is $\rchi(\Pi^*(\qt{\Sigma_{n-1}}{\Sigma_n}))=(-1)^{n-1}(n-1)!$,
  and the right hand side is $(-1)^{n-1}(n-1)!$ as $\mu(1,1)=1$.

 \noindent  \eqref{lemma:key3} This is an immediate consequence of
 \eqref{lemma:key1} and \eqref{lemma:key2}.
\end{proof}

\begin{proof}[Proof of Theorem~\ref{thm:key}]on 
 Combine the expression
 \begin{equation*}
    \rchi_r(\Pi^*(\qt{\Sigma_{n-1}}{\Sigma_{n}}),\Sigma_n) = 
     \frac{1}{n!} \sum_{\substack{ [A \leq \Sigma_n] \\ \text{$A$ free and
           abelian} }} 
    \rchi(\Pi^*(\qt{\Sigma_{n-1}}{\Sigma_{n}})^A) \varphi_r(A) 
  |\Sigma_n : N_{\Sigma_n}(A)|
 \end{equation*}
 for the $r$th equivariant \Euc\ with
 Lemma~\ref{lemma:key}.\eqref{lemma:key3}. Note also that any abelian group of
 order dividing $n$ is realizable as a unique subgroup conjugacy class in
 the symmetric group $\Sigma_n$ acting freely on $\qt
 {\Sigma_{n-1}}{\Sigma_n}$.
This gives
  \begin{equation*}
    \rchi_r(\Pi^*(\qt{\Sigma_{n-1}}{\Sigma_{n}}),\Sigma_n) = 
    -\frac{1}{n}\sum_{|A| \mid n} (-1)^{n/|A|} \mu(1,A) 
    \frac{\varphi_r(A)}{|\Aut{}A|}
  \end{equation*}
  where the sum ranges over the set of isomorphism classes of abelian groups
  $A$ of order dividing $n$.  The \Mb\ function $\mu(1,A)$ is completely known
  \cite[2.8]{hall36}. Indeed, write $A=\prod A_p$ as the product of its \syl
  ps $A_p$. Then $\mu(1,A)=\prod \mu(1,A_p)$ and $\mu(1,A_p)=0$ unless $A_p$
  is an elementary abelian $p$-group. For an elementary abelian $p$-group of
  rank $d$,
\begin{equation*}
  \mu(1,C_p^d) = (-1)^d p^{\binom{d}{2}}
\end{equation*}
Suppose now that $A= \prod A_p$ where each \syl p\ $A_p=C_p^{d_p}$ is
elementary abelian of rank $d_p$.  By \cite[Lemma 2.1]{hillar_rhea}, $\Aut{}A
= \prod_p \Aut{}{A_p} = \prod_p \GL{d_p}p{}$ and clearly $\varphi_r(\prod A_p)
= \prod \varphi_r(A_p)$.  The number of surjections of $C_p^r$ onto $C_p^d$ is
\begin{equation*}
  \varphi_r(C_p^d) = \binom{r}{d}_p |\GL dp{}|
\end{equation*}
and consequently
\begin{equation*}
   \frac{\varphi_r(C_p^d)}{|\Aut{}{C_p^d}|} =
    \binom{r}{d}_p
  \end{equation*}
This finishes the proof.
\end{proof}


Let $c_r(n)=(a \ast b_r)(n)$ denote Dirichlet convolution of the
multiplicative arithmetic function $a(n)$ and $b_r(n)$. The function $a$ is
$-1$ ($+1$) on any even (odd) prime power and the multiplicative function
$b_r$ has value
\begin{equation}\label{eq:brpe}
  b_r(p^e) = (-1)^e p^{\binom e2} \binom{r}{e}_p
\end{equation}
on any prime power $p^e$.
\begin{prop}\label{eq:brrecur}
The multiplicative arithmetic sequences $b_r$ are given by $b_1=\mu$ and 
 the recurrence relations
\begin{equation*}
  b_{r+1}(p^d) = p^db_r(p^d) - p^{d-1}b_r(p^{d-1})
\end{equation*}  
valid for all $r \geq 1$ and all prime powers $p^d$, $d \geq 0$.
\end{prop}
\begin{proof}
  Use Pascal's identities for ordinary and Gaussian binomial coefficients
  \cite[Equation 17b]{stanley97}
\begin{equation*}
  \binom d2 = \binom{d-1}2+(d-1), \qquad
  \binom{r+1}d_p = p^d\binom rd_p + \binom r{d-1}_p
\end{equation*}
and the definition \eqref{eq:brpe} of $b_r$.  
\end{proof}

In the following proposition, $1$ is the constant sequence with value $1$ on
all $n \geq 1$.

\begin{cor} \label{cor:dirichletb}
$(1 \ast b_{r+1})(n) = nb_r(n)$ for all $r,n \geq 1$.
 \end{cor}
 \begin{proof}
   The telescopic sum
   \begin{equation*}
     (1 \ast b_{r+1})(p^d) =  \sum_{e=0}^d b_{r+1}(p^e) =
    \sum_{e=0}^d (p^eb_r(p^e) - p^{e-1}b_r(p^{e-1}) 
   \end{equation*}
 evaluates to $p^db_r(p^d)$ at any prime power $p^d$.
 \end{proof}

\begin{prop} \label{prop:crbr} The multiplicative arithmetic sequences $c_r$
  are given by $c_1=1,-2,0,0,\ldots$ and
  \begin{equation*}
 c_{r+1}(n)= n(b_r(n)-b_r(n/2)) \qquad
  \text{(where $b_r(n/2)=0$ for odd $n$)}
  \end{equation*}
 for all $r,n \geq 1$.
\end{prop}
\begin{proof}
  The two multiplicative sequences $c_1 = a \ast \mu$ and $1,-2,0,0,\ldots$
  are identical since they agree on all prime powers. 
  For {\em odd\/} $n$, $c_{r+1}(n)=(a \ast b_{r+1})(n) = (1 \ast b_{r+1})(n) =
  nb_r(n)$ by Corollary~\ref{cor:dirichletb}. For powers of $2$,
  \begin{equation*}
  c_{r+1}(2^d) = (a \ast b_{r+1})(2^d) = 
  b_{r+1}(2^d) - \sum_{e=0}^{d-1}b_{r+1}(2^e) =
  2^db_r(2^d) - 2^{d-1}b_r(2^{d-1}) -  2^{d-1}b_r(2^{d-1}) =
  2^d(b_r(2^d)-b_r(2^{d-1}))
\end{equation*}
by the recurrence relation of Proposition~\ref{eq:brrecur}.  Thus
$c_{r+1}(n)= n(b_r(n)-b_r(n/2))$ for even $n$ by multiplicativity.
\end{proof}






The multiplicative sequences $c_r$ can be defined recursively. The initial
sequence is $c_1 = 1,-2,0,0,\ldots$. For $r \geq 1$,
\begin{equation*}
 c_{r+1}(2^d) =
 \begin{cases}
  2c_r(2) & d= 1 \\   
   2^dc_r(2^d) + \sum_{j=2}^d 2^{d+j-2}c_r(2^{d-j}) & d \geq 2
 \end{cases}
\end{equation*}
for powers of $2$.  At powers of an {\em odd\/} prime $p$, $c_{r+1}(p^d) =
p^dc_r(p^d)-p^{d-1}c_r(p^{d-1})$ as the sequences $b_r$ and $c_r$ coincide and
we can refer to Proposition~\ref{eq:brrecur}.

\begin{cor}\label{cor:dirichlet}
  The Dirichlet series of the multiplicative arithmetic functions $b_r$ and
  $c_r$ are
  \begin{equation*}
    \sum_{n=1}^\infty \frac{b_r(n)}{n^s} =
    \frac{1}{\zeta(s)\zeta(s-1) \cdots \zeta(s-r+1)}, \qquad
    \sum_{n=1}^\infty \frac{c_r(n)}{n^s} =
    \frac{2^s-2}{2^s\zeta(s-1) \cdots \zeta(s-r+1)}
  \end{equation*}
  where $\zeta(s)$ is the Riemann $\zeta$-function and $r \geq 1$.
\end{cor}
\begin{proof}
  Write $\beta_r(s)$ for the Dirichlet series of $b_r(n)$.
  Corollary~\ref{cor:dirichletb} implies the recurrence
  \begin{equation*}
    \zeta(s)\beta_{r+1}(s) = \beta_r(s-1)
  \end{equation*}
  as $nb_r(n)$, with series $\beta_r(s-1)$, is the Dirichlet convolution of
  $1$, with series $\zeta(s)$, and $b_{r+1}(n)$. (The Dirichlet series of a
  Dirichlet convolution is the product of the Dirichlet series of the
  factors.)  The expression for the Dirichlet series of $b_r(n)$ follows by
  induction starting with the series, $\zeta(s)^{-1}$, for $b_1=\mu$.  The
  Dirichlet series of the Dirichlet convolution $c_r = a \ast b_r$ is the
  product of this series and the series, $\zeta(s) (1-2^{1-s})$, of $a=1 \ast
  c_1$.
\end{proof}

It is easy to make explicit computations on a computer.  The values of the
multiplicative arithmetic function $\frac{1}{n}c_r(n) = \rchi_r(\Pi^*(\qt
{\Sigma_{n-1}}{\Sigma_n}),\Sigma_n)$, $2 \leq n \leq 15$ and $1 \leq r \leq
5$, are
\begin{center}
 \begin{tabular}[t]{>{$}c<{$}|*{14}{>{$}c<{$}}}
\frac{1}{n}c_r(n) &
n=2 & 3 & 4 & 5 & 6 & 7 & 8 & 9 &10 & 11 &12 &13 &14 &15  \\ \hline 
r=1 & 1 & -1 & 0 & 0 & 0 & 0 & 0 & 0 & 0 & 0 & 0 & 0 & 0 & 0 \\
r=2 & -2  & -1  & 1  & -1  & 2  & -1  & 0  & 0  & 2  & -1  & -1  & -1 &2 & 1 \\
r=3 &
-4 & -4 & 5 & -6 & 16 & -8 & -2 & 3 & 24 & -12 & -20 & -14 & 32 & 24   \\
r=4 &
-8 & -13 & 21 & -31 & 104 & -57 & -22 & 39 & 248 & -133 & -273 & -183 & 456 &
403  \\
r=5 &
-16 & -40 & 85 & -156 & 640 & -400 & -190 & 390 & 2496 & -1464 & -3400 & -2380
& 6400 & 6240 
\end{tabular}
\end{center}

\section*{Acknowledgments}
I would like to thank Micha\l\ Adamaszek, Magdalena Kedziorek, Matthew Gelvin,
and Morten S.\, Risager for inspiring discussions and valuable input.


\def\cprime{$'$} \def\cprime{$'$} \def\cprime{$'$} \def\cprime{$'$}
  \def\cprime{$'$}
\providecommand{\bysame}{\leavevmode\hbox to3em{\hrulefill}\thinspace}
\providecommand{\MR}{\relax\ifhmode\unskip\space\fi MR }
\providecommand{\MRhref}[2]{%
  \href{http://www.ams.org/mathscinet-getitem?mr=#1}{#2}
}
\providecommand{\href}[2]{#2}

\end{document}